\documentclass[12pt]{amsart}
\usepackage{fullpage}
\usepackage{tabu}
\usepackage{amssymb} 
\usepackage{amsmath} 
\usepackage{amscd}
\usepackage{amsbsy}
\usepackage{comment, enumerate}
\usepackage[matrix,arrow]{xy}
\usepackage[
        colorlinks, citecolor=blue,
]{hyperref}
\usepackage{mathrsfs}
\usepackage{color}
\usepackage{mathtools,caption}
\usepackage{todonotes}
\usepackage{stmaryrd}
\usepackage{tikz-cd} 
\usepackage[normalem]{ulem} 
\usepackage{cancel}

\numberwithin{equation}{section}

\newtheorem{lemma}{Lemma}[section]
\newtheorem{theorem}[lemma]{Theorem}

\newtheorem{proposition}[lemma]{Proposition}
\newtheorem*{theorem*}{Theorem}

\theoremstyle{definition}

\theoremstyle{remark}
\newtheorem{remark}[lemma]{Remark}

\newcommand{\ZZ}{\mathbb{Z}}
\newcommand{\QQ}{\mathbb{Q}}

\newcommand{\cD}{\mathcal{D}}
\newcommand{\cP}{\mathcal{P}}

\newcommand{\bQQ}{\bar{\mathbb{Q}}}
\newcommand{\OK}{\mathcal{O}_K}
\newcommand{\Fl}{\mathbb{F}_{\ell}}
\newcommand{\bF}{\bar{\mathbb{F}}}

\newcommand{\ty}{\tilde{y}}

\newcommand{\h}{\mathrm{h}}
\newcommand{\fn}{\mathfrak{n}}

\newcommand{\rhoE}{\rho_{E,n}}
\newcommand{\rhof}{\rho_{f,\fn}}
\newcommand{\rhofn}{\rho_{f,n}}

\DeclareMathOperator{\Norm}{\mathcal{N}}
\DeclareMathOperator{\Gal}{Gal}
\DeclareMathOperator{\GL}{GL}
\DeclareMathOperator{\Tr}{Tr}
\DeclareMathOperator{\Frob}{Frob}


\date{\today}

\begin{document}
\title{Sum of consecutive powers as a perfect power}

\author{Angelos Koutsianas}
\address{Department of Mathematics, Aristotle University of Thessaloniki\\
School of Science, 3rd floor, office 17, 54124, Thessaloniki, Greece} 
\email{akoutsianas@math.auth.gr}

\author{Nikos Tzanakis}
\address{Department of Mathematics \& Applied Mathematics \\
University of Crete, Heraklion, Crete, Greece} 
\email{tzanakis@uoc.gr}

\begin{abstract}
In this paper we study the equation
\begin{equation*}
    x^k + (x+1)^k = y^n,\quad n\geq 3,
\end{equation*}
when $k\equiv 2\pmod{4}$. We prove that the only solutions are for $x=0, -1$ when $6\leq k\leq 100$ or for a $k$ with odd prime factors congruent to $3\pmod{4}$. We use linear forms in logarithms, the modular method and the resolution of Thue equations. 
\end{abstract}

\maketitle

\section{Introduction}

A fundamental question in Diophantine equations is the determination of the sums of consecutive perfect powers that are a perfect power. For example, Euler \cite[art. 249]{Euler59} showed the very beautiful relation
\begin{equation*}
    6^3 = 3^3 + 4^3 + 5^3.
\end{equation*}
In the last few decades, many mathematicians have tried to get similar to Euler's above relation, with a variety of different techniques, by studying the equation
\begin{equation}\label{eq:main_general}
    x^k + (x+ 1)^k + \cdots + (x + d-1)^k = y^n,\quad n\geq 2,
\end{equation}
for different values of $k$ and $d$.

When $k=2$ the cases $2\leq d\leq 10$ have been solved by Cohn, Patel and Zhang \cite{Cohn96, Zhang14, Patel18}. The case $k=3$ and $2\leq d\leq 50$ has been resolved by many \cite{Uchiyama79, Cassels85, Stroeker95, Zhang14, BennettPatelSiksek17}. For $k\geq 4$ less are known. The case $k=4$ and $d=2$ is due to the deep results of Ellenberg about the image of the Galois representations of $\QQ$-curves over imaginary quadratic fields and the resolution of the generalized Fermat equation of signature $(2,4,n)$ \cite{Ellenberg04, BennettEllenbergNg10}. For $k=4,5,6$ and $d=3$ we have the results of Zhang and Bennett-Patel-Siksek \cite{Zhang14, BennettPatelSiksek16}. Finally, for higher values of $k$, the asymptotic theorem of Patel and Siksek states that equation \eqref{eq:main_general} has no solutions for almost all $d\geq 2$ when $k$ is even \cite{PatelSiksek17}. We recommend the survey \cite{CoppolaCurcoKhawajaPatelUlkem24} for a detailed exposition of the known results about \eqref{eq:main_general} and its generalizations.

Apart from the cases $k=2,3,4$ there are no results for $d=2$. This is not a coincidence and has a very simple but deep explanation. For $d=2$ the equation \eqref{eq:main_general} is a special case of the generalized Fermat equation of signature $(r,r,p)$. From Darmon's program \cite{Darmon00} we know the existence of the obstructions $(0, 1, 1)$ and $(1, 0, 1)$ that correspond to non-singular Frey hyperelliptic curves with complex multiplication. We could isolate the two obstructions if there was an analogous to Darmon-Merel results \cite{DarmonMerel97} for curves of genus greater than $1$ but this is far beyond the current technology. The same obstructions, and their difficulties, appear in our case for $x=0, -1$ and $|y|=1$.

In this paper, we consider the case $d=2$ and study the equation 
\begin{equation}\label{eq:main}
    x^k + (x+1)^k = y^n, \quad n\geq 3, 
\end{equation}
for $k\equiv 2\pmod{4}$. To the best of our knowledge, it is the first attempt to solve \eqref{eq:main_general} for $d=2$ and big values of $k$. Even though the mentioned obstructions still exist, we succeed in overcoming all of the above difficulties by reducing the resolution of \eqref{eq:main} to the resolution of a finite number of Lebesgue-Nagell equations that we can solve. The main result of the paper is the following theorem.

\begin{theorem}\label{thm:main}
Suppose $k\equiv 2\pmod{4}$. If $6\leq k\leq 100$, then the only 
solutions of the Diophantine equation
\begin{equation} \label{eq:initial}
    x^k + (x+1)^k = y^n, \quad n\geq 3,
\end{equation}
result when $x\in\{-1,0\}$.
\end{theorem}
Note that, if $(x,y,n)$ is a solution with $x\le -2$, then 
$(|x|-1,y,n)$ is also a solution, therefore we may assume that
$x\ge 1$ and, consequently, $y>1$ and odd.
By changing the variable $x$ (and leaving unaltered $y$) we can express the consecutive positive integers 
$x,x+1$ as $(x-1)/2, (x+1)/2$, where $x$ is an odd integer with $x > 1$. Therefore, our equation becomes
\begin{equation}\label{eq:main_symmetric}
   \frac{(x-1)^k + (x+1)^k}{2} = 2^{k-1}y^n, \quad 
   x,y\ge 3\text{ odd},\quad n\ge 3,
\end{equation}
which is symmetric on the left-hand side with respect to $x$ around $0$. 

For the rest of the paper, we solve \eqref{eq:main_symmetric} instead of \eqref{eq:main} for $n=4$ or $n$ is an odd prime. In the study of \eqref{eq:main_symmetric} we use linear forms in logarithms to get an upper bound $n_0$ of $n$, the modular method to prove that there are no other solutions of \eqref{eq:main_symmetric} apart from those with $y=1$ and the resolution of Thue equations when $n\leq 7$.

It is important to mention that the case $k=2$ of \eqref{eq:main_symmetric} has been studied in \cite{Cohn96, MuriefahLucaSiksekTengely09} and we have the following result which we also use for $k\geq 6$.
\begin{theorem}\label{thm:main_k2}
The only solutions of the Diophantine equation
\begin{equation}\label{eq:main_k2}
    x^2 + 1 = 2y^n,\quad n\geq 3,
\end{equation}
are $(|x|, y, n)=(1, 1, n),~(239, 13, 4)$.
\end{theorem}

Theorem \ref{thm:main_k2} and 
Lemma \ref{lem:factorization_without_condition} imply the 
theorem below, whose proof is given immediately after the proof of
Lemma \ref{lem:factorization_without_condition}.
\begin{theorem}\label{thm:main_infinite_family}
If $k\equiv 2\pmod{4}$ with all its odd prime factors  
congruent to $3\pmod{4}$, then the solutions of the  Diophantine 
equation
\begin{equation}\label{eq:main_infinite_family}
    x^k + (x+1)^k = y^n, \quad n\geq 3, 
\end{equation}
 are $(x, y, n, k) = (0, 1, n, k),~(119, 13, 4, 2),~(-120, 13, 4, 2)$.
\end{theorem}

The paper is organized as follows. In Section \ref{sec:modular_method} we recall the basic terminology and results about the modular method. In Section \ref{sec:preliminaries} we prove the key Lemma \ref{lem:equations_system} which shows that \eqref{eq:main_symmetric} is reduced to a finite number of systems of two equations with the first one being a Lebesgue-Nagell equation. At the end of this section we also prove Theorem \ref{thm:main_infinite_family}. In Section \ref{sec:small_n} we show that, for fixed $n$, the resolution of \eqref{eq:main_symmetric} is reduced to the resolution of a finite number of Thue equations, which we solve for $n\leq 7$. In Section \ref{sec:k_big_values} we apply the modular method and prove that there are no solutions for $n\leq n_0$ where $n_0$ is given in Table \ref{table:bound_n}. Finally, in Section \ref{sec:upper_bound_n} we use linear forms in logarithms to prove that there are no solutions with $y>1$ when $n > n_0$.

The computations were carried out on a machine running Ubuntu 24.04.4 with 
an Intel Core i5-10400, $12$ core CPUs at 2.90GHz and 24 GB RAM. For the 
parallel computations we used the $10$ cpus. The code and details how to 
rerun the computations can be found in 
\href{https://github.com/akoutsianas/SumOfConsecutivePowers}{https://github.com/akoutsianas/SumOfConsecutivePowers}.

\begin{quote}
  \emph{For the remainder of the paper, we implicitly assume that} $k\equiv 2\pmod{4}$. 
\end{quote}

\section{Acknowledgments}

The first author was supported by the Special Account for Research Funds AUTH research grant ``\textit{Solving Diophantine Equations-10349}''.

\section{Modular Method}\label{sec:modular_method}

The proof of Theorem \ref{thm:main} for $k\geq 3$ is based, to a
large extent, on the modular method. We use Frey-Hellegouarch curves, modularity, Galois representations and level lowering \cite{Ribet90, Wiles95, TaylorWiles95, BreuilConradDiamondTaylor01, BennettSkinner04, BugeaudMignotteSiksek06}. In this section, we briefly recall terminology and results about the modular method that we use in Section \ref{sec:k_big_values}.

Suppose $f$ is a cuspidal newform of weight $2$, trivial character and level $N_f$ with $q$-expansion
\begin{equation*}
    f = q + \sum_{i=2}^{\infty} a_i(f)q^i.
\end{equation*}
We denote by $K_f$ the \textit{eigenvalue field} of $f$ and say that $f$ is \textit{irrational} if $[K_f:\QQ]>1$ and \textit{rational} otherwise. For a rational prime $n$ we denote by $\fn$ a prime ideal of $K_f$ above $n$. We associate to $f$ a continuous, semisimple Galois representation
\begin{equation*}
    \rhof:~\Gal(\bQQ/\QQ)\rightarrow\GL_2(\bF_n),
\end{equation*}
that is unramified at all primes $\ell\nmid nN_f$ and 
$\Tr(\rhof(\Frob_{\ell}))\equiv a_\ell(f)\pmod{\fn}$, 
where $\Frob_{\ell}$ is a Frobenius element at $\ell$. 
If $f$ is rational we use the notation $\rhofn$ and denote $E_f$ an elliptic 
curve over $\QQ$ corresponding to $f$ due to modularity.

Suppose $E$ is an elliptic curve over $\QQ$ with conductor $N_E$. For a prime 
$\ell \nmid N_E$, let $a_{\ell}(E)=\ell + 1 - \#\tilde{E}(\Fl):=a_{\ell}(\tilde{E})$, where $\tilde{E}$ is the reduction of $E$ at $\ell$. We also denote by $\rhoE$ the Galois representation of $\Gal(\bQQ/\QQ)$ on the $n$-th torsion subgroup of $E$.

\begin{proposition}\label{prop:elimination_step}
With the above notation we assume:  $\rhoE\simeq \rhof$ for a prime ideal $\fn\mid n$ of $K_f$ and $l$ is a prime. Then, the following hold:
\begin{enumerate}[(i)]
    \item if $\ell\nmid nN_EN_f$, then $a_\ell(f)\equiv a_\ell(E)\pmod{\fn}$,
    \item if $\ell\nmid nN_f$ and $\ell\| N_E$, then $\ell + 1\equiv\pm a_\ell(f)\pmod{\fn}$.
\end{enumerate}
If $f$, as above, is rational, then we have
\begin{enumerate}[(i)]
    \item if $\ell\nmid N_EN_{E_f}$, then $a_\ell(E_f)\equiv a_\ell(E)\pmod{n}$,
    \item if $\ell\nmid N_f$ and $\ell\| N_{E_f}$, then $\ell + 1\equiv\pm a_\ell(f)\pmod{n}$,
\end{enumerate}
where $E_f$ is an associate to $f$ elliptic curve over $\QQ$ due to modularity and $N_{E_f}=N_f$ is the conductor of $E_f$.
\end{proposition}
The above proposition is standard (see \cite[Proposition 4.3]{BennettSkinner04} or \cite[p. 203]{Serre87}; in the case that $f$ is rational, see \cite[Lemma 6.4]{BugeaudMignotteSiksek06} or \cite[Proposition 3]{KrausOesterle92}), and we apply it to bound $n$, when $a_\ell(f)- a_\ell(E)\neq 0$, using a 
range of different primes $\ell$. 

\begin{remark}
We recall that, when $f$ is irrational, there are infinitely many primes $\ell$ 
such that $a_{\ell}(f)\not\in\QQ$, which implies that there are infinitely 
many primes $\ell$ for which $a_\ell(f)- a_\ell(E)\neq 0$. Hence, for $f$ 
irrational Proposition \ref{prop:elimination_step} always succeeds.
\end{remark}

\section{Preliminaries}\label{sec:preliminaries}

Define $g_k(t)$ by
\begin{equation}\label{eq:gk}
    g_k(t):=\frac{(t-1)^k + (t+1)^k}{2}.
\end{equation}
The following is a key lemma.
\begin{lemma}\label{lem:quadratic_factor}
We have $g_k(t)\in\ZZ[t]$ and 
\begin{equation*}
    g_k(t) = (t^2 + 1)f_k(t^2),
\end{equation*}
where $f_k(t)\in\ZZ[t]$ and $\deg(f)=\frac{k}{2} - 1$. In particular,
\begin{equation}\label{eq:fk}
    f_k(t) = \sum_{i=0}^{(k/2-1)/2}2^{k/2-(2i+1)}
    \binom{k/2}{2i+1}(t+1)^{2i}t^{(k/2-(2i+1))/2}.
\end{equation}
\end{lemma}

\begin{proof}
The fact that $g_k(t)\in\ZZ[t]$ is an immediate consequence of the definition of $g_k(t)$ in \eqref{eq:gk}. Furthermore, an elementary computation shows that $g_k(\sqrt{-1})=0$.

Since $k\equiv 2\pmod 4$, we express $g_k(t)$ as follows. 
\begin{align*}
    g_k(t) & = \frac{1}{2}\left((t^2+1-2t)^{k/2}+(t^2+1+2t)^{k/2}\right) \\
    & =\frac{1}{2}\sum_{i=0}^{k/2}\binom{k/2}{i}(t^2+1)^i2^{k/2-i}
        (1+(-1)^{k/2-i})t^{k/2-i}.
\end{align*}
For even $i$ we have $1+(-1)^{k/2-i}=0$, therefore,  
\begin{align*}
    g_k(t) & = 
    \frac{1}{2}\sum_{i=0}^{(k/2-1)/2}
    \binom{k/2}{2i+1}(t^2+1)^{2i+1}2^{k/2-i}2t^{k/2-(2i+1)} \\
    & =(t^2+1)\sum_{i=0}^{(k/2-1)/2}2^{k/2-(2i+1)}
    \binom{k/2}{2i+1}(t^2+1)^{2i}t^{k/2-(2i+1)}.
\end{align*}
Because $k/2$ is odd we have $(k/2-(2i+1))/2\in\ZZ$.
\end{proof}
The following lemma will be used in Sections \ref{sec:k_big_values} and 
\ref{sec:upper_bound_n}.
\begin{lemma}\label{lem:equations_system}
Let $k\ge 6$ and suppose $(x,y,n)$ is a solution of \eqref{eq:main_symmetric} with $n>2v_p(k)$ for every prime $p\equiv 1\pmod 4$ dividing $k$. Let $\cP$ be the set of primes which divide $k/2$ and are congruent to $1\pmod 4$, if any such primes exist; otherwise set $\cP=\emptyset$. Define $\cD$ to be the set consisting of $1$ and the positive divisors of $k/2$ whose prime divisors belong to $\cP$ (if $\cP\ne\emptyset$). Then
\begin{align}
    x^2 + 1 & = 2\frac{d_2}{d_1}y_1^n,     \label{eq:express_x^2+1}
    \\
    f_k(x^2) & = 2^{k-2}\frac{d_1}{d_2}y_2^n,   \label{eq:express_f_k(x^2)}
\end{align}
where $d_1,d_2\in\cD$ are relatively prime, $y_1,y_2$ are relatively prime positive odd integers and $y_1y_2=y$.
\end{lemma}
\begin{proof}
By \eqref{eq:main_symmetric} and  Lemma \ref{lem:quadratic_factor} we have $(x^2+1)f_k(x^2)=2^{k-1}y^n$. Since $x$ is odd, $v_2(x^2+1)=1$, therefore $(x^2+1,f_k(x^2)):=2d$ with $d$ odd. If $d>1$, every prime $p$ dividing $d$ also divides $x^2+1$, therefore $p\equiv 1\pmod 4$. Also,
\begin{equation} \label{eq:eq-system-1}
    \frac{x^2+1}{2d}\cdot\frac{f_k(x^2)}{2d}=2^{k-3}\frac{y^n}{d^2}.
\end{equation}
From \eqref{eq:fk} with $t=x^2$ we get $2d\mid 2^{k/2-1}(k/2)$, hence $d\mid (k/2)$; we have already seen that, if $d>1$, every prime divisor of $d$ is $\equiv 1\pmod 4$, therefore $d\in \cD$. Moreover, $v_p(d)\le v_p(k)<n/2$ for every odd prime $p$, a fact that will be used below.

If $d=1$, then we put $d_1=d_2=1$ and \eqref{eq:eq-system-1} 
immediately implies the relations \eqref{eq:express_x^2+1} and
\eqref{eq:express_f_k(x^2)} with appropriate $y_1,y_2$ as in
the announcement of the lemma.

We consider now the case $d>1$.
From the last displayed equation it is clear that $y^n/d^2\in\ZZ$. 
Let $\cP_0\neq\emptyset$ be the set of primes dividing $d$ 
($\cP_0\subseteq\cP$), and write 
$d=\prod_{p\in \cP_0}p^{e_p}$ for its canonical factorization. 
Then $y=y'\prod_{p\in \cP_0}p^{m_p}$  where $m_p:=v_{p}(y)>0$ and  
$(y',\prod_{p\in \cP_0}p)=1$. Consequently, by \eqref{eq:eq-system-1},
\begin{equation} \label{eq:eq-system-2}
    \frac{x^2+1}{2d}\cdot\frac{f_k(x^2)}{2d}= 
 2^{k-3}y'^n\prod_{p\in \cP_0}p^{nm_p-2e_p}
=2^{k-3}\prod_{p\in \cP_0}p^{-2e_p}\cdot
\left(y'\prod_{p\in \cP_0}p^{m_p}\right)^n,
\end{equation}
Note that, for every $p\in \cP_0$, the $p$-exponent on the right-hand side is positive because $m_p>0$ and $n>2v_p(d)=2e_p$. On the other hand, the two factors on the left-hand side are relatively prime, therefore, for every $p\in \cP_0$, precisely one of them has positive $p$-exponent. It follows that we can set  $\cP_0=\cP_1\sqcup \cP_2$, where $\cP_1$ is the set of those primes $p\in \cP_0$ with non-zero $p$-exponent in $(x^2+1)/(2d)$ and, analogously, $\cP_2$ is the set of the primes $p\in \cP_0$ whose $p$-exponent in $f_k(x^2)/(2d)$ is non-zero. Note that $\cP_1$ or $\cP_2$ (but not both) may be empty. Now, from \eqref{eq:eq-system-2} it follows that 

\begin{equation}  \label{eq:express_the_2_factors}
\frac{x^2+1}{2d}=
\prod_{p\in \cP_1}p^{-2e_p}\cdot
\left(y'_1\prod_{p\in \cP_1}p^{m_p}\right)^n,
\quad
 \frac{f_k(x^2)}{2d} = 2^{k-3}
    \prod_{p\in \cP_2}p^{-2e_p}\cdot
    \left(y'_2\prod_{p\in \cP_2}p^{m_p}\right)^n
\end{equation}
with $(y'_1,y'_2)=1$, $y'_1$ odd and $y'_1y'_2=y'$.
Putting
\begin{equation*}
    y_i = y'_i\prod_{p\in \cP_i}p^{m_p},\quad i=1,2,
\end{equation*}
we see that $y_1$ is odd and
$y_1y_2=y'_1y'_2\prod_{p\in \cP_0}p^{m_p}=y'\prod_{p\in \cP_0}p^{m_p}=y$.
Then, by \eqref{eq:express_the_2_factors},
\begin{equation*}
x^2+1=2d\prod_{p\in\cP_1}p^{-2e_p}y_1^n,\quad
f_k(x^2)=2^{k-3}d\prod_{p\in\cP_2}p^{-2e_p}y_1^n,
\end{equation*}
and, since $d=\prod_{p\in\cP_1\sqcup\cP_2}p^{e_p}$, it is 
straightforward to see that
\begin{equation*}
x^2+1=2\prod_{p\in\cP_1}p^{-e_p}\prod_{p\in\cP_2}p^{e_p}y_1^n,
\quad
f_k(x^2)= 2^{k-3}\prod_{p\in\cP_1}p^{e_p}\prod_{p\in\cP_2}p^{-e_p}y_2^n.
\end{equation*}
We complete the proof by setting 
\begin{equation*}
d_1= \prod_{p\in\cP_1}p^{e_p},\quad d_2 = \prod_{p\in\cP_2}p^{e_p}.
\end{equation*}
\end{proof}
\begin{remark} \label{rem:n>2v_p(k)}
For $6\leq k\leq 100$ we have $v_p(k)=1,2$ for every odd prime factor $p$ of $k$. Therefore, the condition $n>2v_p(k)$ holds for any $n \geq 5$.
\end{remark}
As an application of Lemma \ref{lem:equations_system} we obtain Table \ref{table:d1_d2_values} which will be used in Section \ref{sec:k_big_values}.
\begin{table}[h]
    \centering
    \begin{tabular}{| c | c || c | c | }
        \hline
        $k$ &  $(d_1, d_2)$ & $k$ & $(d_1, d_2)$\\
        \hline
        $10$ & $(1, 5), (5, 1)$ & $70$ &  $(1, 5), (5, 1)$  \\ 
        \hline 
        $26$ &  $(1, 13), (13, 1)$ & $74$ &  $(1,37), (37, 1)$ \\
        \hline
        $30$ &  $(1, 5), (5, 1)$ & $78$ & $(1, 13), (13, 1)$\\
        \hline
        $34$ &  $(1, 17), (17, 1)$ & $82$ & $(1, 41), (41, 1)$\\
        \hline
        $50$ &  $(1, 5), (1, 5^2), (5, 1), (5^2, 1)$ 
        & $90$ & $(1, 5), (5, 1)$\\
        \hline
         $58$ &  $(1, 29), (29, 1)$ &  & \\
         \hline
    \end{tabular}
    \caption{All possible values of the pairs $(d_1,d_2)$ for $6\leq k\leq 100$, satisfying the necessary conditions of Lemma \ref{lem:equations_system} and $d_1d_2>1$.} \label{table:d1_d2_values}
\end{table}

The next lemma proves a ``detail'' that is necessary for the arguments of 
Section \ref{sec:upper_bound_n}.
\begin{lemma}\label{lem:y1_gt_3}
With the assumptions and notations of Lemma \ref{lem:equations_system}, if
$6\leq k\leq 100$ then, in \eqref{eq:express_x^2+1} we have $y_1\ge 3$.
\end{lemma}
\begin{proof}
From Remark \ref{rem:n>2v_p(k)} the condition $n>v_p(k)$ is satisfied for all prime factors $p\equiv1\pmod{4}$ of $k$ and $6\leq k\leq 100$. From Lemma \ref{lem:equations_system} we know that $y_1$ is odd. 

Suppose that $y_1=1$. This implies that $x^2 + 1 = 2d_2/d_1$ where $(d_1,d_2)$ are as in Table \ref{table:d1_d2_values} or $d_1=d_2=1$. Since $x>1$, the second alternative is excluded. Because $(d_2, d_1)=1$ and both are odd, we conclude that $d_1=1$, hence $x^2 + 1=2d_2$ with $d_2>1$. It follows that $2d_2-1$ is a square, which occurs exactly in the following cases
\begin{align*}
    (d_2,k,x) = & (5,10,3), (13,26,5), (5,30,3), (5,50,3), (25,50,7), (5,70,3),  \\
    & (13,78,5), (41,82,9), (5, 90,3).
\end{align*}
For every such triple we check whether the integer 
$m(k,x):=\frac{1}{2}\left((x-1)^k+(x+1)^k\right)$ is an $n$-th power with 
$n\ge 3$. For every $m(k,x)$ we easily show that this is not the case, by 
finding a prime $q$ such that $v_q(m(k,x))<3$.
\end{proof}

The condition in Lemma \ref{lem:equations_system} that $n > v_k(p)$ 
for all prime divisors $p$ of $k$ with $p\equiv 1\pmod{4}$ is not 
satisfied in general. The next lemma is an analogous and less explicit 
version of Lemma \ref{lem:equations_system} but without any condition 
on $n$. We use it in the proof Theorem \ref{thm:main_infinite_family}
and in Section \ref{sec:small_n}.

\begin{lemma} \label{lem:factorization_without_condition}
If in the assumptions and notations of Lemma \ref{lem:equations_system} we
remove the condition ``$n>v_k(p)$ for every prime $p\equiv 1\pmod 4$ dividing
$k$'', then we have
\begin{equation} \label{eq:x^2+1=ay1^n}
    x^2 + 1=2a\ty_1^{n},
\end{equation}
where $a$ is odd $> 1$ and $\ty_1$ is an odd positive integer. Moreover, every 
prime divisor $p$ of $a$ divides $k/2$, is congruent to $1\pmod 4$ and satisfies 
$0\leq v_p(a) < n$.
\end{lemma}
\begin{proof}
The proof of Lemma \ref{lem:equations_system} up to relation \eqref{eq:eq-system-1} still holds; note that, in order to obtain \eqref{eq:eq-system-1}, the condition $n>2v_p(k)$ is not used. Here, $\cP$ has the same meaning as in that lemma.    
In \eqref{eq:eq-system-1}, the factor $y^n/d^2$ of the right-hand side is, 
clearly, a positive integer, therefore any prime $p$ dividing $d$ has to divide 
$y$ and the $p$-exponent of $y$ must be appropriately large in order that the 
$p$-exponent of $y^n/d^2$ be non-negative. It follows that 
$y^n/d^2=z^n\prod_{p\in\cP}p^{s_p}$, with $z$ a positive integer (not 
necessarily prime to $\prod_{p\in\cP}p$) and $s_p\ge 0$ for every $p\in\cP$. 
From the fact that the two factors in the left-hand side of 
\eqref{eq:eq-system-1} are relatively prime with the first of them being odd, it 
follows that $(x^2+1)/(2d)=z_1^n\prod_{p\in\cP_1}p^{s_p}$, where $\cP_1$ is a 
subset (probably empty) of $\cP$ and $z_1$ is a positive divisor of $z$. 
Consequently,
\begin{equation*}
x^2+1=2z_1^n\prod_{p\in\cP}p^{v_p(d)}\prod_{p\in\cP_1}p^{s_p}
    =2z_1^n\prod_{p\in\cP}p^{e_p},
\end{equation*}
where $e_p=s_p+v_p(d)$ if $p\in\cP_1$ and $e_p=v_p(d)$ otherwise. For every $p\in\cP$ we write $e_p=nm_p+r_p$ with $m_p\ge 0$ and $0\le r_p<n$, so that $\prod_{p\in\cP}p^{e_p}$ is equal to a $n$-th power of an integer times 
\begin{equation*}
    a:=\prod_{p\in\cP}p^{r_p}.
\end{equation*}
This clearly leads to \eqref{eq:x^2+1=ay1^n}. It remains to show that $a>1$. If 
not, then $x^2+1=2\ty_1^n$ and, since $x>1$, Theorem \ref{thm:main_k2} implies
that $x=239$. In the proof of Theorem \ref{thm:main_infinite_family} (see below) 
we show that $(239,y,n)$ cannot be a solution of \eqref{eq:main_symmetric}, and 
this completes our proof.
\end{proof}

\begin{proof}[Proof of Theorem \ref{thm:main_infinite_family}]
Let $(x,y,n)$ be a solution of \eqref{eq:main_infinite_family}
with $x\ne 0$.
If $x<0$, we see that the equation is equivalent to 
$x_1^k+(x_1+1)^k=y^n$ with $x_1=|x|-1$. Therefore it suffices 
to consider \eqref{eq:main_infinite_family} with $x>1$ and prove
that then its only solution is $(x,y,n,k)=(119,13,4,2)$.

By Lemma \ref{lem:factorization_without_condition} and the hypothesis that $k/2$ has no 
prime factor congruent to $1\pmod{4}$, the resolution of 
\eqref{eq:main_infinite_family} is reduced to the 
resolution of \eqref{eq:main_k2}. If $n$ has an odd prime factor, 
then Theorem \ref{thm:main_k2} implies that $x=1$, which we reject. 
It remains to consider the case $n=4$. For this case, the same 
theorem forces $x=239$ and, going back to \eqref{eq:main_symmetric}, 
we get
\begin{equation*}
    \left(\frac{239 - 1}{2}\right)^k + \left(\frac{239 + 1}{2}\right)^k = 119^k+120^k=y^4.
\end{equation*}
On putting $k=2m\geq 6$, with $m>1$ odd, the last equation implies that the triple $(119^m, 120^m,y^2)$ is Pythagorean and primitive, therefore there exist relatively prime integers $u>v>0$ of opposite parity, such that
\begin{equation*}
    u^2-v^2=119^m,\quad 2uv=120^m.
\end{equation*}
The first equation implies that either of the following two cases
holds:
\begin{eqnarray}
     u+v=119^m & \text{\;\;and} & u-v=1, \label{eq:u+v,u-v,case_I}
     \\
     u+v=17^m & \text{\;\; and} & u-v=7^m \label{eq:u+v,u-v,case_II}
\end{eqnarray}
On reducing both \eqref{eq:u+v,u-v,case_I} and \eqref{eq:u+v,u-v,case_II} $\pmod 4$ we see that $u$ is even and $v$ is 
odd, therefore, from the equation $2uv=120^m$ we see that one of the 
following four cases holds:
\begin{equation} \label{eq:uv_cases}
    (u,v) \in \{(2^{3m-1}, 15^m),\: (2^{3m-1}3^m,5^m),\:
    (2^{3m-1}5^m,3^m),\:(2^{3m-1}15^m,1)\}
\end{equation}
Since $u>v$, the first case is rejected. 

Now, assume that \eqref{eq:u+v,u-v,case_I} holds. Then, reducing
$\pmod 8$ the system of the two equations gives 
$(u+v, u-v) \equiv (-1,1)\pmod 8$, hence $v\equiv -1\pmod 4$, which
excludes the second and fourth case \eqref{eq:uv_cases}.
The third case of \eqref{eq:uv_cases} remains. Combining this with
\eqref{eq:u+v,u-v,case_I} gives $(119^m+1)/2=2^{3m-1}5^m$ and
$(119^m-1)/2=3^m$, from which it follows that 
$2^{3m-1}5^m=3^m+1$. Since $m$ is odd $>1$, the left-hand side is 
$\equiv 0\pmod 8$, while the right-hand side is $\equiv 4\pmod 8$,
a contradiction.  

Next, assume that \eqref{eq:u+v,u-v,case_II} holds. Reduction
$\pmod 8$ shows that $v\equiv 1\pmod 4$, which excludes the third
case \eqref{eq:uv_cases}. Considering \eqref{eq:u+v,u-v,case_II}
$\pmod{16}$, we see that $v\equiv 5\pmod 8$, which excludes the
fourth case \eqref{eq:uv_cases}. The second case, combined with
\eqref{eq:u+v,u-v,case_II}, gives $17^m=2^{3m-1}3^m+5^m$, which is
impossible $\pmod 8$ because $m$ is odd $>1$.
\end{proof}

\section{The case of small $n$} \label{sec:small_n}

Let $k\ge 6$ and suppose $(x,y,n)$ is a solution of \eqref{eq:main_symmetric}. 
From Lemma \ref{lem:factorization_without_condition} we know
\begin{equation} \label{eq:x^2+1=ay1^n-again}
    x^2 + 1=2a\ty_1^n,
\end{equation}
where $\ty_1$ is an odd positive integer, $a$ is odd $>1$ and every prime divisor $p$ of $a$ divides $k/2$, is congruent to $1\pmod 4$ and satisfies $0\leq v_p(a) < n$. For a given $n$, the resolution of \eqref{eq:x^2+1=ay1^n-again} will be reduced to that of a finite number of Thue equations. 

For this section, let $i:=\sqrt{-1}$ and $K=\QQ(i)$, with maximal order $\OK=\ZZ[i]$ in which unique factorization holds. For $z\in K$ we denote by $\bar{z}$ the complex-conjugate of $z$. 

Factorizing \eqref{eq:x^2+1=ay1^n-again} over $\OK$ we get
\begin{equation*}
    (x + i) (x - i)=2a\ty_1^n.
\end{equation*}
We have $(x + i, x - i) = 1 + i$ because $x$ is odd; also,
$2=-i(1 + i)^2$. Therefore, from the above equation we have
\begin{equation*}
    \frac{x+i}{1+i}\cdot\frac{x-i}{1-i}=
    \frac{x+i}{1+i}\cdot \frac{x-i}{-i(1+i)} =a\ty_1^n,
\end{equation*}
where the two factors in the left-hand side are relatively prime.
Since $a, \ty_1, x$ are odd and all prime factors of $a$ and  $\ty_1$ are congruent to $1\pmod{4}$ (implying that every such prime is the product of two complex-conjugate factors) it follows that 
\begin{equation*}
    \frac{x+i}{1+i} = i^{\nu}\alpha (t+si)^n,
\end{equation*}
where $\alpha\in\OK$, $\alpha\bar{\alpha}=a$, $s,t\in\ZZ$ with $(t+si)(t-si)=y_1$ and $0\le\nu\le 3$. Therefore, 
\begin{align*}
x+i & = i^{\nu}(1+i)\alpha (t+si)^n, \\ 
x-i & =i^{-\nu}(1-i)\bar{\alpha}(t-si)^n.
\end{align*}
Subtracting these equalities 
and dividing through by $i$ we get
\begin{align*}
    2 & =i^{\nu-1}(1+i)\alpha (t+si)^n -
    i^{-\nu-1}(1-i)\bar{\alpha}(t-si)^n 
      \\
    & = i^{\nu-1}(1+i)\alpha (t+si)^n +
   i^{-(\nu-1)}(1-i)\bar{\alpha}(t-si)^n
     \\
   & = \beta (t+si)^n +\bar{\beta}(t-si)^n
\end{align*}
where $\beta=i^{\nu-1}(1+i)\alpha$ and $\beta\bar{\beta}=2a$.
Thus,
\begin{equation}   \label{eq:Thue_and_y1}
    \Re\left(\beta (t+si)^n\right)=1\;\;\text{and}\;\; \ty_1=t^2+s^2.
\end{equation}
For $2<k\le 98$ with at least one prime divisor congruent to $1\pmod 4$, we have
$\cP=\{p\}$, where $p$ is given in the table below, and
$p=\pi\bar{\pi}$ with the prime ideals $\pi\OK$, $\bar{\pi}\OK$ 
being distinct; moreover, if $\pi\mid\alpha$, then 
$\bar{\pi}\nmid\alpha$ (otherwise $a$ would divide $x+i$, which is 
absurd). In the Table below we see, for every $k$ as above, 
a prime element $\pi$ dividing $p$ and the value of 
$\beta=(1+i)\pi$ which, of course, satisfies $\beta\bar{\beta}=2a$.
\begin{table}[h]
    \centering
    \begin{tabular}{|c||c|c|c|c|c|c|c|c|c|c|c|} \hline
        $k$ & $10$ & $26$ & $30$ & $34$ & $50$ & $58$ & $70$
            & $74$ & $78$ & $82$ & $90$
          \\  \hline
        $p$ & $5$ & $13$ & $5$ & $17$ & $5$ & $29$ & $5$ 
        & $37$ & $13$ & $41$ & $5$
        \\ \hline
       $\pi$ & $2+i$ & $3+2i$ & $2+i$ & $4+i$ & $2+i$ &  $5+2i$ & 
       $2+i$ & $6+i$ & $3+2i$ & $5+4i$ & $2+i$ 
       \\ \hline
       $\beta$ & $1+3i$ & $1+5i$ & $1+3i$ & $3+5i$ & $1+3i$ &
       $3+7i$ & $1+3i$ & $5+7i$ & $1+5i$ & $1+9i$ & $1+3i$
       \\ \hline
       $2a$ & $10$ & $26$ & $10$ & $34$ & $10$ & $58$ & $10$ & $74$ &
       $26$ & $82$ & $10$ 
       \\ \hline
    \end{tabular}
    \label{tab:a-val}
\end{table}

\noindent
An important remark is that, in order to check whether there exist 
integer pairs $(t,s)$ satisfying \eqref{eq:Thue_and_y1}, we do 
not need to  consider either $\bar{\beta}$ (because 
$\Re\left(\bar{\beta}(t+si)^n\right)
           =\Re\left(\beta (t-si)^n\right)$), 
or any other associate of $\beta$; it suffices to take $\beta$
from the Table and consider the four equations
\begin{equation}  \label{eq:Thue-small_n}
    \Re\left(\beta (t+si)^n\right)=\pm 1, \quad
    \Im\left(\beta (t+si)^n\right) = \pm 1.
\end{equation}
Indeed, clearly, 
$\Re\left(-\beta (t+si)^n\right)=-\Re\left(\beta (t+si)^n\right)$;
also, by elementary computations we can see that
\begin{equation*}
    \Re\left(\pm i \beta (t+si)^n\right) =
    \begin{cases}
  \mp \Im\left(\beta (s-ti)^n\right) & \mbox{if $n\equiv 0\pmod 4$,}     \\
  \mp \Re\left(\beta (s-ti)^n\right) & \mbox{if $n\equiv 1\pmod 4$,}
   \\
  \pm \Im\left(\beta (s-ti)^n\right) & \mbox{if $n\equiv 2\pmod 4$,}
  \\
 \pm \Re\left(\beta (s-ti)^n\right) & \mbox{if $n\equiv 3\pmod 4$.} 
    \end{cases}
\end{equation*}
Any of the four relations \eqref{eq:Thue-small_n} leads to a $n$-degree Thue equation if the left-hand side is irreducible; if not it leads to an equation of the form $h_1(t,s)h_2(t,s)=1$, where $h_1,h_2$ are polynomials with rational integer coefficients, from which $h_1(t,s)=h_2(t,s)=1$ or $h_1(t,s)=-h_2(t,s)=1$ and such systems of equations are easily solved. 

Concerning the $k$-values under consideration and $n=3,\ldots,7$, the left-hand sides of all equations \eqref{eq:Thue-small_n} are irreducible and we solve them using, for example, Pari/GP \cite{pari24}. It turns out that all of them, with two exceptions, are either impossible, or their solutions satisfy $st=0$.
If $st=0$, then $\ty_1=1$, implying $x^2+1=2a$. From the Table above we see that
$2a-1$ is a square in the following cases: 
$(k,x)=(10,3), (26,5), (30,3), (50,3), (70,3), (78,5), (82,9)$, $(90,3)$.
For every such pair we check whether the integer
$m(k,x):=\frac{1}{2}\left((x-1)^k+(x+1)^k\right)$ is an $n$-th power with 
$n\ge 3$. For every $m(k,x)$ we show easily that this is not the case, 
by finding a prime $q$ such that $v_q(m(k,x))<3$.
The exceptional cases (solutions $(s,t)$ with $st\ne 0$) occur when 
$k\in\{26,78\}$ and $n=3$; then equations \eqref{eq:Thue-small_n} furnish also 
the solutions $(t,s)\in\{\pm (-1,-2),\,\pm (-2,3)\}$, which give $\ty_1=5$ or 
$13$. An easy machine-check shows that the equation \eqref{eq:main_symmetric} 
with $\ty_1\in\{3,5\}$, $n=3$ and $k\in\{26,78\}$ has no integer solutions $x$. 
Thus, we have the following result:
\begin{proposition}\label{prop:small_exponents}
Equation \eqref{eq:main_symmetric} is impossible for $n\leq 7$ and 
$6\leq k\leq 100$.
\end{proposition}
\section{The case $6\leq k\leq 100$}\label{sec:k_big_values}

Let $k\ge 6$ and suppose $(x,y,n)$ is a solution of \eqref{eq:main_symmetric} with prime $n\ge 11$. Then from Lemma \ref{lem:equations_system} we know that a solution $(x,y)$ of \eqref{eq:main_symmetric} corresponds to a solution of the equation
\begin{equation}\label{eq:d2_d1_y1n}
    x^2 + 1=2\frac{d_2}{d_1}y_1^n,
\end{equation}
where $d_1$, $d_2$ and $y_1$ as in Lemma \ref{lem:equations_system}. The case $d_2=d_1=1$ is excluded by Theorem \ref{thm:main_k2}. Therefore, for the rest of this section we assume $d_1d_2>1$. The possible values of $d_1$ and $d_2$ for $6\leq k\leq 100$ are presented in Table \ref{table:d1_d2_values}.

For the resolution of \eqref{eq:d2_d1_y1n} we apply the modular method as it has been developed in \cite{BennettSkinner04}. We attach the Frey-Hellegouarch curve
\begin{equation*}
    E:~Y^2=X^3 + 2xX^2 + 2\frac{d_2}{d_1}y_1^nX=X^3 + 2xX^2 + (x^2 + 1)X.
\end{equation*}
The invariants of the elliptic curve $E$ have been computed in \cite[Lemma 2.1]{BennettSkinner04}.
\begin{equation*}
\begin{aligned}
    j(E) & =-2^6\frac{d_1^2(x^2 - 3)^3}{(d_2y_1^n)^2}, & \Delta(E)& = -2^8\frac{d_2^2}{d_1^2}y_1^{2n}, &  
    N(E)& =2^7\prod_{p\mid d_1d_2y_1}\!\!p,\\
    c_4(E) & = 2^4(x^2 - 3), &  c_6(x) &= 2^6x(x^2 + 9).
\end{aligned}
\end{equation*}
We denote by $\rhoE$ the Galois representation of $\Gal(\bar{\QQ}/\QQ)$ acting 
on the $n$-th torstion points $E[n]$ of $E$. Since $\rhoE$ is absolutely 
irreducible \cite[Corollary 3.1]{BennettSkinner04}, by level lowering 
\cite{Ribet90} and modularity 
\cite{Wiles95, TaylorWiles95, BreuilConradDiamondTaylor01} 
we get the following lemma.

\begin{lemma}\label{lem:newform_isomorphism}
There exists a newform $f$ of weight $2$, trivial character and level $N_n(E)=2^7\prod_{p\mid d_1d_2}p$ such that $\rhoE\simeq\rhof$, where $\fn\mid n$ is a prime ideal in $K_f$.
\end{lemma}

\begin{proposition}\label{prop:modular_method_k}
Suppose $6\leq k\leq 100$ and $d_1d_2>1$. If $k\in\{34,74\}$, then there are not solutions. Suppose $k\neq 34, 74$:
\begin{itemize}
    \item if $k\neq 58$, then $\rhoE\simeq\rhofn$\,;
    \item if $k=58$, then $n=13$ or $\rhoE\simeq\rhofn$, 
\end{itemize}
where, in both cases, $f$ is a rational newform of weight $2$, trivial 
character and level $N_n(E)$ which lies in an explicitly computable set.
\end{proposition}

\begin{proof}
For all different values of $6\leq k\leq 100$ appearing in Table \ref{table:d1_d2_values}, Lemma \ref{lem:newform_isomorphism} implies that we have $\rhoE\simeq\rhof$, where $f$ is a newform of weight $2$, trivial character and level $N_n(E)$ belonging to an explicit finite set $\mathcal{F}$. For each $f\in\mathcal{F}$ we try to exploit Proposition \ref{prop:elimination_step}, by choosing primes $\ell< 50$, in order to bound $n$. If $k\in\{34,74\}$ or $k\ne 58$ and Proposition \ref{prop:elimination_step} succeeds for $f$, it turns out that $n\le 7$.

If $k=58$, all but four newforms $f\in\mathcal{F}$ for which Proposition \ref{prop:elimination_step} succeeds, give us $n\le 7$. Concerning the four exceptional newforms, although they are irrational (which, in general, is better suited to the application of Proposition \ref{prop:elimination_step}) we could not do better than $n=13$, even when we also tried with several primes $\ell>50$.

For the newforms $f$ the Proposition \ref{prop:elimination_step} fails, all are rationals and we have $\rhoE\simeq\rhofn$. The total amount of time was 
approximately $9$ minutes, including the time for the computation of the newforms.
\end{proof}

Let $\mathcal{F}_0\subseteq\mathcal{F}$ be the set of the newforms that we cannot eliminate using Proposition \ref{prop:elimination_step}; in particular, $\mathcal{F}_0$ contains the four irrational newforms which we encounter when $k=58$ and $n=13$. All the remaining newforms in $\mathcal{F}_0$ are rational.

\subsection{Eliminate newforms for given $n$}

Our previous discussion implies that there are values of $k$ and newforms $f$ for which Proposition \ref{prop:elimination_step} either gives a difference $a_\ell(f)- a_\ell(E)= 0$ for every prime $\ell$ that we tried or $k=58$ and $n=13$. These newforms $f$ form the set $\mathcal{F}_0$, defined just above. In this subsection, we show how we can eliminate these newforms for a given $n$. Since, for every $k$ under consideration, Table \ref{table:bound_n} gives us an upper bound $n_0$ of $n$, we can finish the proof of Theorem \ref{thm:main}. The method that we apply is a modification of the method in \cite[Section 8]{BugeaudMignotteSiksek06} and \cite{Kraus98}.

Suppose $\ell= tn+1$ is a prime with $n\geq 11$ and $\ell\nmid k$. We define
\begin{align*}
    \mu_t(\Fl) & =\big\{\zeta\in\Fl^\times:~\zeta^t=1\big\},\\ 
    A(t,\ell) & =\big\{(x_0, \zeta)\in\Fl\times\mu_t(\Fl): ~\frac{\overline{d_1}}{\overline{2d_2}}(x_0^2+1)=\zeta,~\frac{\overline{d_2}f_k(x_0^2)}{\overline{2^{k-2}d_1}}\in\mu_t(\Fl)\cup\{0\}\big\}.
\end{align*}
where $\bar{d}_i$ is the image of $d_i$ in $\Fl$ for $i=1,2$. For a pair $(x_0,\zeta)\in A(t,\ell)$ we define the elliptic curve
\begin{equation*}
    E_{x_0,\zeta}:~Y^2=X^3 + 2x_0X^2+2\bar{d}_2\bar{d}_1^{-1}\zeta X= X^3 + 2x_0X^2+(x_0^2+1)X.
\end{equation*}

\begin{proposition}\label{prop:eliminate_fix_n}
Suppose $f$ is a newform as in Lemma \ref{lem:newform_isomorphism} for a putative solution of the system in Lemma \ref{lem:equations_system}. 
Suppose that there exists $t\geq 2$ such that
\begin{enumerate}
    \item $\ell=tn+1$ is a prime and $\ell\nmid k$,
    \item For all $(x_0,\zeta)\in A(t,\ell)$ we have
    \begin{equation*}
        \begin{cases}
            n\nmid\Norm_{K_f/\QQ}((a_\ell(f) - a_{\ell}(E_{x_0,\zeta})\cdot(a_\ell(f)^2 - 4)), & \text{if } \ell\equiv 1\pmod{4}, \\
            n\nmid\Norm_{K_f/\QQ}(a_\ell(f) - a_{\ell}(E_{x_0,\zeta})), & \text{if } \ell\equiv 3\pmod{4}.
        \end{cases}
    \end{equation*}
\end{enumerate}
Then, the system in Lemma \ref{lem:equations_system} does not have any solutions for the given $n$ arising from the newform $f$. 
\end{proposition}

\begin{proof}
Let $(x_1,y_1,n)$ be a solution of the system in Lemma \ref{lem:equations_system}. Because $\ell\nmid k$ it follows that $\ell\nmid N_n(E)$. 

First, suppose that $\ell\nmid y_1$. Reducing the system in Lemma \ref{lem:equations_system} modulo $\ell$ we see that there exists $(x_0, \zeta)\in A(t,\ell)$ such that $\bar{x}= x_0$ and $\bar{y}_1^n= \zeta$. Then, from the definition of $E_{x_0,\zeta}$ we see that this is the reduction
of the elliptic curve $E$ modulo $\ell$, hence $a_{\ell}(E)=a_{\ell}(E_{x_0,\zeta})$. 
Applying Lemma \ref{lem:newform_isomorphism} and 
Proposition \ref{prop:elimination_step} 
we get that 
\begin{equation*}
n\mid \Norm_{K_f/\QQ}(a_\ell(f) - a_{\ell}(E_{x_0,\zeta})),
\end{equation*}
which contradicts the second property of $\ell$.

Finally, if $\ell\mid y_1$, it should hold $\ell\equiv 1\pmod{4}$ from \eqref{eq:express_x^2+1}. From Proposition \ref{prop:elimination_step} we understand that $a_{\ell}(f)\equiv \pm (\ell + 1)\equiv \pm 2\pmod{\fn}$ since $\ell\equiv 1\pmod{n}$. Therefore, we get 
\begin{equation*}
n\mid \Norm_{K_{f}/\QQ}(a_{\ell}(f)^2 - 4) 
\end{equation*}
which contradicts the first property of $\ell$.
\end{proof}

\begin{proposition}\label{prop:no_solutions_in_range}
Suppose $k$, $d$ and $(d_1,d_2)$ are as in Table \ref{table:d1_d2_values} for $k\neq 34, 74$. Then, there are no solutions to the system of Lemma \ref{lem:equations_system} for $11\leq n\leq n_0$ where $n_0$ is as in Table \ref{table:bound_n}.
\end{proposition}

\begin{proof}
We have written a Sage script that applies Proposition \ref{prop:eliminate_fix_n} for all possible values of $d_1$, $d_2$ and $k$ as in Table \ref{table:d1_d2_values}, all newforms $f\in\mathcal{F}_0$ that have not been eliminated in Proposition \ref{prop:modular_method_k} and all $11\leq n \leq n_0$.

We have also applied Proposition \ref{prop:eliminate_fix_n} for $k=58$ and those newforms that have been eliminated by Proposition \ref{prop:elimination_step} but the exponent $n=13$ has not been excluded. For all cases, it was enough to consider $t\leq 1050$. The running time of the computations for all $k$'s was approximately $9$ hours.
\end{proof}

\section{Upper bound for $n$}\label{sec:upper_bound_n}
We use linear forms in logarithms to get an upper bound of $n$ for the values of $k$ in Table \ref{table:d1_d2_values}. Throughout this section, $(x,y)$ is a solution of \eqref{eq:main_symmetric} and $X=x^2+1$.  

From Lemma \ref{lem:equations_system} we have 
\begin{equation} \label{eq:x^2+1_and_fk(x^2)}
\begin{aligned}
    x^2 + 1 & = 2\frac{d_2}{d_1}y_1^n, 
    \\
    f_k(x^2) & = 2^{k-2}\frac{d_1}{d_2}y_2^n,
\end{aligned}
\end{equation}
where $d_1,d_2$ are relatively prime positive odd divisors of $k/2$ whose prime divisors -if any at all- are congruent to $1\pmod{4}$. From Lemma \ref{lem:y1_gt_3} we have $y_1\geq 3$.

We set 
\begin{equation*}
    f_k(X-1)=X^{k/2-1} + h_{k/2-2}X^{k/2-2} + \cdots + h_0
\end{equation*}
and it is easily checked from \eqref{eq:fk} that
\begin{equation*}
h_{k/2-2}= k\left(\frac{k}{2}-1\right).
\end{equation*}
Then,
\begin{align*}
1 + \frac{h_{k/2-2}}{X} + \cdots + \frac{h_0}{X^{k/2-1}} 
&  = \frac{2^{k-2}(d_1/d_2)y_2^n}{X^{k/2-1}} =
2^{k-2}\frac{\frac{d_1}{d_2}y_2^n}{\left(2\frac{d_2}{d_1}y_1^n\right)^{k/2-1}} \\
& =2^{k/2-1}
\left(\frac{d_1}{d_2}\right)^{k/2}\left(\frac{y_2}{y_1^{k/2-1}}\right)^n = A\left(\frac{y_2}{y_1^{k/2-1}}\right)^n,
\end{align*}
where
\begin{equation*} 
    A:= 2^{k/2-1}
    \left(\frac{d_1}{d_2}\right)^{k/2}
    \ne 1.
\end{equation*}
We define 
\begin{equation*} 
    X_0:=
    \max\left\{\left(\sum_{i=0}^{k/2-3}|h_i|\right)/h_{k/2-2},\:
    200h_{k/2-2}\right\}+1.
\end{equation*}

\begin{lemma}
Suppose $X < X_0$. Then,
\begin{equation*}
    n < \frac{\log\left(\frac{kX_0}{4}\right)}{\log(3)} := n_1.
\end{equation*}
\end{lemma}

\begin{proof}
We have
\begin{equation*}
    X_0 > X=x^2+1 = 2\frac{d_2}{d_1}y_1^n\geq \frac{4}{k} 3^n,
\end{equation*}
since $d_2\geq 1$, $d_1\leq \frac{k}{2}$ and $y_1\geq 3$. Last inequality gives the bound of $n$.
\end{proof}

In what follows, we assume that $X\ge X_0$. Thus,
\begin{equation*}
    \left|A\left(\frac{y_2}{y_1^{k/2-1}}\right)^n-1\right|
      =
    \left|\frac{h_{k/2-2}}{X} + \frac{h_{k/2-3}}{X^2} 
    +\cdots +\frac{h_0}{X^{k/2-1}}\right|
    \le \frac{h_{k/2-2}}{X}
    +\frac{1}{X}\sum_{i=0}^{k/2-3}\frac{|h_i|}{X^{k/2-2-i}}.
\end{equation*}
We have 
\begin{equation*}
\sum_{i=0}^{k/2-3}\frac{|h_i|}{X^{k/2-2-i}} =
\frac{1}{X}\sum_{i=0}^{k/2-3}\frac{|h_i|}{X^{k/2-3-i}}
\le \frac{1}{X_0}\sum_{i=0}^{k/2-3}|h_i|\le |h_{k/2-2}|,
\end{equation*}
with the last inequality holding because $X\geq X_0$. Consequently,
\begin{equation*}  
     \left|A\left(\frac{y_2}{y_1^{k/2-1}}\right)^n-1\right|
     \le  \frac{2h_{k/2-2}}{X} 
     =\frac{d_1h_{k/2-2}}{d_2}\cdot\frac{1}{y_1^n}.
\end{equation*}
Note that $2h_{k/2-2}/X<0.01$ because $X\geq X_0$.

We apply \cite[Lemma B.2]{SmartBook98} with $\Delta =A\left(y_2/y_1^{k/2-1}\right)^n$ and $a=0.01$, to obtain
\begin{equation} \label{eq:bnd_logDelta}
   |\log\Delta | < 1.0051|\Delta -1|\leq 1.0051\cdot
   \frac{d_1h_{k/2-2}}{d_2}\cdot\frac{1}{y_1^n} 
   =\frac{c_1}{y_1^n},
\end{equation}
where
\begin{equation}\label{eq:c1}
c_1:=1.0051\cdot\frac{d_1h_{k/2-2}}{d_2}
 \le 0.50255k^2\left(\frac{k}{2}-1\right),
\end{equation}
because $d_1/d_2\le k/2$. 

\begin{lemma} \label{lem:Laurent_ini_conds}
Let
\begin{equation*}
n_2:=\left\lfloor\max\left\lbrace
\frac{\log\left(\frac{c_1}{|\log A|}\right)}{\log 3}\,,\,
\frac{\log c_1}{\log 3}
\,,\,\frac{|\log A|+1}{\log 1.5}
\right\rbrace\right\rfloor + 1
\end{equation*}
and
\begin{equation*}
 \Lambda =n\log\alpha_1 -\log\alpha_2,
 \end{equation*}
 where
\begin{equation*}
(\alpha_1,\alpha_2) =
\begin{cases}
  \left(y_2/y_1^{k/2-1}\,,\, 1/A\right) &
  \mbox{if $A<1$},
  \\[10pt]
  \left(y_1^{k/2-1}/y_2\,,\, A\right) &
  \mbox{if $A>1$}.
\end{cases}
\end{equation*}
Then, for $n>n_2$ we have $\alpha_i>1$  ($i=1,2$) and 
$\displaystyle{  |\Lambda | <\frac{c_1}{y_1^n}}$.
Moreover, $\alpha_1<3/2$.
\end{lemma}
\begin{proof}
Let $\Delta = A\left(y_2/y_1^{k/2-1}\right)^n$, as in 
\eqref{eq:bnd_logDelta}. Also, note that, in both cases,
$\log\alpha_2=|\log A|$.

If $A<1$, we claim that $y_2/y_1^{k/2-1}>1$. Indeed, assuming the contrary, we have $|\log\Delta| = |n\log(y_1^{k/2-1}/y_2)+\log(1/A)| $ and the numbers inside the logarithms are $>1$. Therefore, 
\begin{equation*}
|\log\Delta| = n\log(y_1^{k/2-1}/y_2)+\log(1/A)\geq \log(1/A).
\end{equation*}
Then, by \eqref{eq:bnd_logDelta}, 
\begin{equation*}
\log(1/A)\le c_1/y_1^n\le c_1/3^n,
\end{equation*} 
which contradicts  $n>n_2$. Consequently, 
$|\log\Delta| = |n\log (y_2/y_1^{k/2-1})-\log(1/A)| = |n\log\alpha_1-\log\alpha_2|$ $ = |\Lambda|$ and $\alpha_i>1$ for $i=1,2$.

If $A>1$ we work similarly. Now we claim $y_1^{k/2-1}/y_2>1$. Assuming the contrary, we have  $|\log\Delta| = |n\log(y_2/y_1^{k/2-1})+\log A|$ and the numbers inside the logarithms are $>1$. Therefore, 
\begin{equation*}
|\log\Delta| = n\log(y_2/y_1^{k/2-1})+\log A\geq \log A.
\end{equation*}
Then, by \eqref{eq:bnd_logDelta}, 
\begin{equation*}
\log A\le c_1/y_1^n\le c_1/3^n,
\end{equation*}
which contradicts our assumption $n>n_2$. Consequently, 
$|\log\Delta| = |n\log (y_1^{k/2-1}/y_2)-\log A|$ $=|n\log\alpha_1-\log\alpha_2|=|\Lambda|$ and $\alpha_i>1$ for $i=1,2$.

Finally, we show that $\alpha_1<3/2$. Indeed, if $\alpha_1\ge 3/2$, then $|\Lambda|=\Lambda$ (because $n> \frac{|\log A|}{\log 1.5}$), thus, by \eqref{eq:bnd_logDelta} and $n>n_2$ we have, 
\begin{equation*}
\frac{c_1}{3^n}\ge\frac{c_1}{y_1^n}>|\Lambda| =\Lambda =
n\log\alpha_1-\log\alpha_2\ge n\log 1.5 -\log\alpha_2,
\end{equation*}
hence, 
\begin{equation*}
n<\frac{1}{\log 1.5}\left(\frac{c_1}{3^n}+\log\alpha_2\right)=
\frac{1}{\log 1.5}\left(\frac{c_1}{3^n}+|\log A|\right)
<\frac{1+|\log A|}{\log 1.5}
\end{equation*}
which contradicts $n>n_2$.
\end{proof}
\begin{remark} \label{rem:n1}
For all values $6\le k\le 100$ and every $d_1/d_2$, easy computations show that $n_2\le 554$.
\end{remark}
Now we apply \cite[Corollary 2]{Laurent08} and the notation
therein, in order to obtain a lower bound for $\log|\Lambda|$. We use Lemma \ref{lem:Laurent_ini_conds}. In our case $\alpha_1,\alpha_2\in\QQ$, therefore $D=1$.

From Lemma \ref{lem:Laurent_ini_conds} and its proof we notice that if $A<1$, then $y_2/y_1^{k/2-1}=\alpha_1 < 3/2$, hence $y_2 < 1.5y_1^{k/2-1}$; and if $A>1$, then $1 < \alpha_1=y_1^{k/2-1}/y_2$, hence $y_2<y_1^{k/2-1}$.
Therefore, 
\begin{equation*}
\h(\alpha_1)=\log\left(\max\{y_1^{k/2-1},y_2\}\right)<\log(1.5y_1^{k/2-1}).
\end{equation*}
Also, 
\begin{equation*}
\h(\alpha_2)=\log\left(\max\{2^{k/2-1}d_1^{k/2},d_2^{k/2}\}\right)\leq \frac{k}{2}\log k,
\end{equation*}
from the fact $d_i\leq k/2$. Thus, in the notation of \cite[Corollary 2]{Laurent08}, we choose
\begin{equation*}
  \log A_1:=\log\left(1.5y_1^{k/2-1}\right), \quad \log A_2:=\frac{k}{2}\log k.
\end{equation*}
Also,
\begin{equation*}
b':=\frac{b_1}{D\log A_2} + \frac{b_2}{D\log A_1} =
\frac{2n}{k\log k} +\frac{1}{\log(1.5y_1^{k/2-1})} 
<\frac{2.01n}{k\log k},
\end{equation*}
with the right-most inequality hold in provided that 
\begin{equation*}
n>100\frac{k\log k}{\log\left(1.5\cdot 3^{k/2-1}\right)}:=n_3.
\end{equation*}
For $m\in\{10,12,14,\ldots,30\}$ and $C_2=C_2(m)$ as in 
\cite[Table 1]{Laurent08}, if 
\begin{equation} \label{eq:lb_n_Laurent_cond}
    n\ge \frac{e^{m-0.38}k\log k}{2}:=n_4,
\end{equation}
then $\log b'+0.38 \ge m$ and now, by \cite[Corollary 2]{Laurent08},
\begin{equation*}
    \log|\Lambda|\ge -C_2(\log b')^2\log A_1\log A_2
    > -C_2\left(\log\left(\frac{2.01n}{k\log k}\right)\right)^2
    \log\left(1.5y_1^{k/2-1}\right)\frac{k}{2}\log k.
\end{equation*}
By Lemma \ref{lem:Laurent_ini_conds} and equations \eqref{eq:bnd_logDelta} and \eqref{eq:c1} we have the inequality
\begin{equation*}
    \log|\Lambda|<\log c_1 -n\log y_1 
    \le \log\left(0.50255k^2(k/2-1)\right)-n\log y_1.
\end{equation*}
Combining the last two displayed relations gives
\begin{equation*}
 n < C_2\left(\log\left(\frac{2.01n}{k\log k}\right)\right)^2
\frac{\log(1.5y_1^{k/2-1})}{\log y_1}\frac{k\log k}{2} +
\frac{\log\left(0.50255k^2(k/2-1)\right)}{\log y_1}.
\end{equation*}
We have $y_1\ge 3$, therefore
\begin{equation} \label{eq:upp_bnd_n}
     n < C_2 \left(\log\left(\frac{2.01n}{k\log k}\right)\right)^2\left(\frac{k}{2}-1+\frac{\log 1.5}{\log 3}\right)\frac{k\log k}{2} + \frac{\log\left(0.50255k^2(k/2-1)\right)}{\log 3}.
\end{equation}
If $n$ satisfies \eqref{eq:lb_n_Laurent_cond} and $n$ is sufficiently large, the relation \eqref{eq:upp_bnd_n} is impossible. For $6\le k\le 66$ we consider \eqref{eq:upp_bnd_n} with $m=10$ and (by \cite[Table 1]{Laurent08}) $C_2=25.2$, while for $70\le k\le 98$ we choose $m=12$ and (by \cite[Table 1]{Laurent08}) $C_2=23.4$. As a consequence, an explicit upper bound $n_0$ for $n$ is obtained and is given in Table \ref{table:bound_n}. For every $k\in\{6,10,14,\ldots,98\}\setminus\{6,10,14,70\}$ the upper bound is larger than $n_4$. For $k=6,10,14,70$, the upper bound resulting from \eqref{eq:upp_bnd_n} is $<n_4$, therefore in Table \ref{table:bound_n} the value  $\lfloor n_4\rfloor$ is given. For every $k$, the bound for $n$ satisfies $n > n_1, n_2,n_3$.

\begin{table}[h]
    \centering
    \begin{tabular}{| c | c || c | c |}
    \hline
        $k$  & $n_0$ & $k$  & $n_0$
    \\ \hline  
        $10$ & $173419$ & $70$ & $16550269$ 
    \\ \hline
        $26$ & $1438387$ & $74$ & $18440172$
    \\ \hline
        $30$ & $2084701$ & $78$ & $20987069$   
    \\ \hline 
        $34$ & $2875394$ & $82$ & $23725616$ 
    \\ \hline
        
      $50$ & $7627398$ & $90$ & $29791351$
    \\ \hline  
     $58$ & $11043693$ & &
     \\ \hline    
    \end{tabular}
    \caption{Upper bound of $n$ for $k$ in Table \ref{table:d1_d2_values}.}
    \label{table:bound_n}
\end{table}

\section{Proof of Theorem \ref{thm:main}}

\begin{proof}
Assume that equation \eqref{eq:initial} has a solution with $x\neq -1,0$. As explained immediately after the announcement of Theorem \ref{thm:main}, the existence of such a solution is equivalent to the existence of a solution of \eqref{eq:main_symmetric} with $x>1$. We show that this leads to a contradiction.

Without loss of generality, we assume that $n$ is prime hence, by Proposition \ref{prop:small_exponents}, $n\ge 11$. On the other hand, in Section \ref{sec:upper_bound_n} we proved that $n\le n_0$, with $n_0$ given in Table \ref{table:bound_n}; thus, $11\le n\le n_0$. Then, by Lemma \ref{lem:equations_system} and Remark \ref{rem:n>2v_p(k)}, the system of the two displayed equations therein has a solution. 

The first equation of this system (equation \eqref{eq:express_x^2+1}) is identical to \eqref{eq:d2_d1_y1n} and, as explained below it, the case $d_1=d_2=1$ is excluded due to Theorem \ref{thm:main_k2}. Therefore, we have $d_1d_2>1$ where the pairs $(d_1,d_2)$ are those appearing in Table \ref{table:d1_d2_values}. If $k=34, 74$ from Proposition \ref{prop:modular_method_k} we have $n\leq 7$ which is a contradiction. If $k\neq 34, 74$, then by Proposition \ref{prop:no_solutions_in_range}, the system of equations of Lemma \ref{lem:equations_system} has no solution for $11\leq n\leq n_0$, which gives the final contradiction.
\end{proof}

\section{Conclusion}

The methods and analysis we use to get Theorem \ref{thm:main} is, in principle, applicable to any given value $k\equiv 2\pmod{4}$ with $k>100$. Therefore, it is quite natural to move our attention to the cases $k\not\equiv 2\pmod{4}$. The case $k\equiv 0\pmod{4}$ is a consequence of the work of Bennett-Ellenberg-Ng about the generalized Fermat equation of signature $(2,4, n)$ that we have already mention \cite{Ellenberg04, BennettEllenbergNg10}. Therefore, the remaining case is when $k$ is odd. 

For $k$ odd the situation is quite different. The key Lemma \ref{lem:equations_system} no longer holds; however, 
equation \eqref{eq:main_symmetric} can be written as 
\begin{equation}
    xf_k(x)=2^{k-1}y^n,
\end{equation}
where $f_k(x)\in\ZZ[x]$. This implies that $x$ is almost a perfect $n$-th power which is a very restricted property for $x$. Even the case $k=5$ is a very challenging problem.

\bibliography{main.bib}{}
\bibliographystyle{plain}
\end{document}